\documentclass[11pt]{amsart}
\usepackage[foot]{amsaddr}
\usepackage[english]{babel}

\allowdisplaybreaks

\usepackage[
pdftex,hyperfootnotes]{hyperref}

\usepackage{nicematrix}
\NiceMatrixOptions{renew-dots,renew-matrix}
\usepackage[utf8]{inputenc}

\usepackage{comment}

\usepackage{amssymb,latexsym,amsmath,amsthm,bm}
\usepackage{mathrsfs}
\usepackage{mathtools,arydshln,mathdots}
\mathtoolsset{showonlyrefs}

\theoremstyle{plain}

\newtheorem{teo}{Theorem}[section]
\newtheorem{coro}[teo]{Corollary}

\newtheorem{pro}[teo]{Proposition}

\theoremstyle{defi}
\newtheorem{defi}[teo]{Definition}

\theoremstyle{remark}

\newcommand{\diag}{\operatorname{diag}}

\newcommand{\C}{\mathbb{C}}

\newcommand{\N}{\mathbb{N}}

\newcommand{\prodint}[1]{\left\langle{#1}\right\rangle}

\allowdisplaybreaks[1]

\makeatletter
\DeclareRobustCommand{\gaussk}{\DOTSB\gaussk@\slimits@}
\newcommand{\gaussk@}{\mathop{\vphantom{\sum}\mathpalette\bigcal@{K}}}

\newcommand{\bigcal@}[2]{%
	\vcenter{\m@th
		\sbox\z@{$#1\sum$}%
		\dimen@=\dimexpr\ht\z@+\dp\z@
		\hbox{\resizebox{!}{0.8\dimen@}{$\mathcal{K}$}}%
	}%
}
\newcommand{\cfracplus}{\mathbin{\cfracplus@}}
\newcommand{\cfracplus@}{%
	\sbox\z@{$\dfrac{1}{1}$}%
	\sbox\tw@{$+$}%
	\raisebox{\dimexpr\dp\tw@-\dp\z@\relax}{$+$}%
}
\newcommand{\cfracdots}{\mathord{\cfracdots@}}
\newcommand{\cfracdots@}{%
	\sbox\z@{$\dfrac{1}{1}$}%
	\sbox\tw@{$+$}%
	\raisebox{\dimexpr\dp\tw@-\dp\z@\relax}{$\cdots$}%
}
\makeatother

\makeatletter
\newcommand*{\relrelbarsep}{.386ex}
\newcommand*{\relrelbar}{%
	\mathrel{%
		\mathpalette\@relrelbar\relrelbarsep
	}%
}
\newcommand*{\@relrelbar}[2]{%
	\raise#2\hbox to 0pt{$\m@th#1\relbar$\hss}%
	\lower#2\hbox{$\m@th#1\relbar$}%
}
\providecommand*{\rightrightarrowsfill@}{%
	\arrowfill@\relrelbar\relrelbar\rightrightarrows
}
\providecommand*{\leftleftarrowsfill@}{%
	\arrowfill@\leftleftarrows\relrelbar\relrelbar
}
\providecommand*{\xrightrightarrows}[2][]{%
	\ext@arrow 0359\rightrightarrowsfill@{#1}{#2}%
}
\providecommand*{\xleftleftarrows}[2][]{%
	\ext@arrow 3095\leftleftarrowsfill@{#1}{#2}%
}
\makeatother

\catcode`,\active

\catcode`\,12

\allowdisplaybreaks[1]
\begin{document}

 \title[Uvarov Perturbations for Matrix Orthogonal Polynomials]{Uvarov Perturbations for Matrix Orthogonal Polynomials}

 \author[G Ariznabarreta]{Gerardo Ariznabarreta}

\author[JC García-Ardila]{Juan C. García-Ardila$^{1,\maltese}$}
\address{$^1$Departamento de Matemática Aplicada a la Ingeniería Industrial, Universidad Politécnica de Madrid, E.T.S. Ingenieros Industriales, José Gutiérrez Abascal 2, 28006 Madrid, Spain}
\email{juancarlos.garciaa@upm.es}

\author[M Mañas]{Manuel Mañas$^{2,\dagger}$}
 \address{$^2$Departamento de Física Teórica, Universidad Complutense de Madrid, Ciudad Universitaria, Plaza de Ciencias 1,  28040 Madrid, Spain}
 \email{manuel.manas@ucm.es}

\author[F Marcellán]{Francisco Marcellán$^{3,\dagger}$}
\address{$^3$Departamento de Matemáticas, Universidad Carlos III de Madrid, Avd/  Universidad 30, 28911 Leganés, Spain}
\email{pacomarc@ing.uc3m.es}

\thanks{$^\maltese$Acknowledges financial support from Comunidad de Madrid multiannual agreement with the Universidad Rey Juan Carlos under the grant Proyectos I+D para Jóvenes Doctores, Ref. M2731, project NETA-MM}
	\thanks{$^\dagger$Acknowledges financial support from Spanish ``Agencia Estatal de Investigación'' research project [PID2021- 122154NB-I00], \emph{Ortogonalidad y Aproximación con Aplicaciones en Machine Learning y Teoría de la Probabilidad}.}

\keywords{Matrix biorthogonal polynomials, spectral theory of matrix polynomials,  quasidefinite matrix of  generalized kernels, Gauss--Borel factorization,  matrix Uvarov transformations, Christoffel type formulas}
\subjclass{42C05,15A23}
\enlargethispage{.13cm}
\begin{abstract}
Additive perturbations, specifically, matrix Uvarov transformations for matrix orthogonal polynomials, are under consideration. Christoffel--Uvarov formulas are deduced for the perturbed biorthogonal families, along with their matrix norms. These formulations are expressed in terms of the spectral jets of the Christoffel--Darboux kernels. 
	\end{abstract}

\maketitle



\section{Introduction}
In this paper, we extend the analysis previously conducted on Christoffel and Geronimus perturbations, as well as their combinations explored in \cite{nuestro0,nuestro1,nuestro2}, by examining the case of additive perturbations to a given matrix of measures and the corresponding matrix orthogonal polynomials.

The polynomial ring $\mathbb C^{p\times p}[x]$ serves as a free bimodule over the ring of matrices $\mathbb C^{p\times p}$ with the canonical basis $\{I_p, I_p x, I_p x^2, \dots\}$, where $I_p\in\C^{p\times p}$ is the identity matrix. Noteworthy free bisubmodules include the sets $\mathbb C_m^{p\times p}[x]$ of matrix polynomials of degree less than or equal to $m$. A basis for $\mathbb C_m^{p\times p}[x]$, with a cardinality of $m+1$, is given by $\{I_p, I_p x, \dots, I_p x^m\}$. Since $\mathbb C$ possesses the invariant basis number (IBN) property, so does $\mathbb C^{p\times p}$ (see \cite{rowen}). Consequently, as $\mathbb C^{p\times p}$ is an IBN ring, the rank of the free module $\mathbb C_m^{p\times p}[x]$ is unique and equals $m+1.$ In other words, any other basis has the same cardinality. The algebraic dual of $\big(\mathbb C_m^{p\times p}[x]\big)^*$ consists of homomorphisms $\phi :\mathbb C_m^{p\times p}[x]\rightarrow \mathbb C^{p\times p}$ that, for the right module, take the form
\begin{align*}
\langle \phi,P(x)\rangle&=\phi_0 p_0+\dots+\phi_m p_m,  & P(x)&=p_0+\dots +p_mx^m,
\end{align*}
where $\phi_k\in\mathbb C^{p\times p}$.
Hence, we can establish an identification between the dual of the right module and the corresponding left submodule. This dual module is a free module with a singular rank, precisely $m+1$, and it possesses a dual basis denoted by $\{(I_p x^k)^*\}_{k=0}^m$. This basis satisfies
$
\langle(I_px^k)^*,I_p x^l\rangle=\delta_{k,l}I_p.
$
Similar statements hold for the left module $\mathbb C_m^{p\times p}[x]$.

 \begin{defi}[Sesquilinear form]\label{def:sesquilinear}
A sesquilinear form $\prodint{\cdot,\cdot}$ defined on the bimodule $\mathbb{C}^{p\times p}[x]$ is a continuous mapping given by
\begin{align*}
	\begin{array}{cccc}
		\prodint{\cdot,\cdot}: &\mathbb{C}^{p\times p}[x]\times\mathbb{C}^{p\times p}[x]&\longrightarrow &\mathbb{C}^{p\times p},\\
		&(P(x), Q(x))&\mapsto& \prodint{P(x),Q(y)}.
	\end{array}
\end{align*}
Here, $\prodint{\cdot,\cdot}$ has the following properties for any triple $P(x),Q(x),R(x)\in  \mathbb{C}^{p\times p}[x]$, with $A,B\in\mathbb{C}^{p\times p}$.
\begin{enumerate}
	\item  $\prodint{AP(x)+BQ(x),R(y)}=A\prodint{P(x),R(y)}+B\prodint{Q(x),R(y)}$,
	\item $\prodint{P(x),AQ(y)+BR(y)}=\prodint{P(x),Q(y)}A^\top+\prodint{P(x),R(y)}B^\top$.
\end{enumerate}
\end{defi}
The reader may have observed that, despite dealing with complex polynomials in a real variable, we have chosen the transpose instead of the Hermitian conjugate. For any pair of matrix polynomials, the sesquilinear form is defined as follows
\begin{align*}
\prodint{P(x),Q(y)}&=\sum_{\substack{k=1,\dots,\deg P\\
		l=1,\dots,\deg Q}}p_k G_{k,l}(q_l)^\top, &P(x)&=\sum\limits_{k=0}^{\deg P}p_kx^k, &Q(x)&=\sum\limits_{l=0}^{\deg Q} q_lx^l,
\end{align*}
where the coefficients are $G_{k,l}=\prodint{I_px^k ,I_py^l }$.
The corresponding semi-infinite matrix
\begin{align*}
	G=\left[\begin{NiceMatrix}
		G_{0,0 } &G_{0,1}& \Cdots\\
		G_{1,0} & G_{1,1} & \Cdots\\
		\Vdots[shorten-end=2pt] & \Vdots[shorten-end=2pt] & 
	\end{NiceMatrix}\right]
\end{align*}
is known in the literature as the Gram matrix of the sesquilinear form.

\begin{defi}
A matrix of generalized kernels
\begin{align*}
	u_{x,y}:=\begin{bmatrix}
		(u_{x,y})_{1,1}& \Cdots &(u_{x,y})_{1,p}\\
		\vdots & & \vdots\\
		(u_{x,y})_{p,1} & \Cdots & (u_{x,y})_{p,p}
	\end{bmatrix},
\end{align*}
where $(u_{x,y})_{k,l}\in(\mathbb C[x,y])'$,  provides a continuous sesquilinear form with entries given by
\begin{align*}
	\big(\langle P(x),Q(y)\rangle_u\big)_{i,j}&=\sum_{k,l=1}^p\big\langle(u_{x,y})_{k,l}, P_{i,k}(x)\otimes Q_{j,l}(y)\big\rangle\\
	&=\sum_{k,l=1}^p\big\langle \mathcal L_{u_{k,l}}(Q_{j,l}(y)),P_{i,k}(x) \big\rangle,
\end{align*}
where $\mathcal L_{u_{k,l}}:\mathbb C[y]\to(\mathbb C[x])'$ is a continuous linear operator. This can be condensed into a matrix form. For $u_{x,y}\in (\mathbb C^{p\times p}[x,y])'=(\mathbb C^{p\times p}[x,y])^*\cong \mathbb C^{p \times p}[\![x,y]\!]$, a sesquilinear form is given by
\begin{align*}
	\langle P(x), Q(y) \rangle_u&=\langle u_{x,y} , P(x)\otimes Q(y)\rangle=\langle\mathcal L_u(Q(y)),P(x)\rangle,
\end{align*}
where $\mathcal L_u: \mathbb C^{p\times p}[y]\to(\mathbb C^{p\times p}[x])'$ is a continuous linear map.
\end{defi}

In the sequel $N_0\coloneq \{0,1,2,\dots\}$.

\begin{defi}[Biorthogonal matrix polynomials]
Given a sesquilinear form $\prodint{\cdot,\cdot}$, two sequences of matrix polynomials $\big\{P_n^{[1]}(x)\big\}_{n=0}^\infty$ and $\big\{P_n^{[2]}(x)\big\}_{n=0}^\infty$ are said to be biorthogonal with respect to $\prodint{\cdot,\cdot}$ if:
\begin{enumerate}
	\item $\deg(P_n^{[1]}(x))=\deg(P_n^{[2]}(x))=n$ for  $n\in\N_0$,
	\item $\prodint{P_n^{[1]}(x),P_m^{[2]}(y)}=\delta_{n,m}H_n$ for  $n,m\in\N_0$,
\end{enumerate}
where $H_n$ are nonsingular matrices, and $\delta_{n,m}$ is the Kronecker delta.
\end{defi}

\begin{defi}[Quasidefiniteness]
	A Gram matrix of a sesquilinear form $\langle\cdot,\cdot\rangle_u$ is said to be quasidefinite if $\det G_{[k]}\neq 0$ for all $k\in\{0,1,\dots\}$.
	Here, $G_{[k]}$ denotes the truncation
	\begin{align*}
		G_{[k]}:=\begin{bmatrix}
			G_{0,0}&\Cdots & G_{0,k-1}\\
			\vdots & & \vdots\\
			G_{k-1,0} & \Cdots &G_{k-1,k-1}
		\end{bmatrix}.
	\end{align*}
	We say that the bivariate generalized function $u_{x,y}$ is quasidefinite and the corresponding sesquilinear form is nondegenerate when its Gram matrix is quasidefinite.
	\end{defi}	

\begin{pro}[Gauss--Borel factorization, see \cite{ari}]\label{pro:fac}
If the Gram matrix of a sesquilinear form $\langle\cdot,\cdot\rangle_u$ is quasidefinite, then a unique Gauss--Borel factorization exists, and it is given by
\begin{align}\label{eq:gauss}
	G=(S_1)^{-1} H (S_2)^{-\top},
\end{align}
where $S_1,S_2$ are lower unitriangular block matrices, and $H$ is a diagonal block matrix
\begin{align*}
	S_i&=\left[\begin{NiceMatrix}
		I_p&0_p&0_p&\cdots\\
		(S_i)_{1,0}& I_p&0_p&\cdots\\
		(S_i)_{2,0}& (S_i)_{2,1}&I_p&\ddots\\
		\Vdots[shorten-end=2pt]	&\Vdots[shorten-end=2pt]&\Ddots[shorten-end=-20pt]&\ddots
	\end{NiceMatrix}\right], &i&=1,2,&
	H&=\diag(
	H_0,H_1, H_2,\dots),
\end{align*}
with $(S_i)_{n,m}$ and $H_n\in\mathbb C^{p\times p}$, $ n,m\in\N_0$. Here $0_p\in\C^{p\times p}$ is the zero matrix.
\end{pro}

	We define  $\chi(x):=\left[\begin{NiceMatrix}
	I_p&I_p x&I_px^2&\Cdots
	\end{NiceMatrix}\right]^\top$.

%
	\begin{defi}\label{defi:bio2kind}
Given a quasidefinite matrix of generalized kernels $u_{x,y}$ and the Gauss--Borel factorization \eqref{pro:fac} of its Gram matrix, the corresponding first and second families of matrix polynomials are
\begin{align}\label{eq:bior}
	P^{[1]}(x)=\begin{bNiceMatrix}
		P^{[1]}_0(x)\\[5pt]P^{[1]}_1(x)\\\Vdots[shorten-end=3pt]
	\end{bNiceMatrix}&\coloneq S_1\chi(x), &  P^{[2]}(y)=\begin{bNiceMatrix}
		P^{[2]}_0(y)\\[5pt]P^{[2]}_1(y)\\\Vdots[shorten-end=3pt]
	\end{bNiceMatrix}&\coloneq S_2\chi(y),
\end{align}
respectively.
\end{defi}

\begin{pro}[Biorthogonality and orthogonality]
Given a quasidefinite matrix of generalized kernels $u_{x,y}$, the first and second families of monic matrix polynomials $\big\{P_n^{[1]}(x)\big\}_{n=0}^\infty$ and $\big\{P_n^{[2]}(x)\big\}_{n=0}^\infty$ are biorthogonal, i.e.,
\begin{align}\label{eq:biorthogonal}
\begin{aligned}
		\prodint{P^{[1]}_n(x),P^{[2]}_m(y)}_u&=\delta_{n,m}H_n,& n,m&\in\N_0.
\end{aligned}
\end{align}
For $m\in\{1,\dots,n-1\}$, the following orthogonality relations hold
\begin{align}
	\label{eq:orthogonality1}\prodint{P^{[1]}_n(x),y^mI_p}_u&=0_p,&\prodint{x^mI_p, P^{[2]}_n(y)}_u&= 0_p,\\
	\label{eq:orthogonality2}\prodint{P^{[1]}_n(x),y^nI_p}_u&= H_n, & \prodint{x^nI_p, P^{[2]}_n(y)}_u&= H_n.
\end{align}
\end{pro}

\begin{defi}[Christoffel--Darboux kernel,\cite{DAS, simon-cd}]
Given two sequences of matrix biorthogonal polynomials
$\big\{P_k^{[1]}(x)\big\}_{k=0}^\infty$ and $\big\{P_k^{[2]}(y)\big\}_{k=0}^\infty$, with respect to the sesquilinear form $\prodint{\cdot,\cdot}_u$, we define the $n$-th Christoffel--Darboux kernel matrix polynomial
\begin{align}\label{eq:CD kernel}
	K_{n}(x,y):=\sum_{k=0}^{n}(P_k^{[2]}(y))^\top( H_k)^{-1}P^{[1]}_k(x).
\end{align}
\end{defi}
\begin{pro}
	\begin{enumerate}
		\item For a quasidefinite matrix of generalized kernels $u_{x,y}$, the corresponding  Christoffel--Darboux kernel gives the projection operator (reproducing property)
		\begin{align}\label{eq:reproducing}
		\prodint{ K_n(x,z),\sum_{0\leq j\ll\infty} C_j P^{[2]}_j(y)}_u&=
	\Big(	\sum_{j=0}^nC_jP_j^{[2]}(z)\Big)^\top,\\	\prodint{ \sum_{0\leq j\ll\infty}C_jP^{[1]}_j(x),(K_n(z,y))^\top}_u&=
		\sum_{j=0}^nC_jP^{[1]}_j(z).
		\end{align}
		\item
		In particular, we have
		\begin{align}\label{eq:K-u}
		\prodint{ K_n(x,z),I_py^l}_u&=I_pz^l, & l\in&\{0,1,\dots,n\}.
		\end{align}
	\end{enumerate}
\end{pro}
\begin{proof}
	It follows directly from the biorthogonality condition  \eqref{eq:biorthogonal}.
\end{proof}

\section{Matrix Uvarov transformations}

Uvarov perturbations for the scalar case, involving a finite number of Dirac deltas, were first considered in \S 2 of \cite{Uva} within the context of orthogonal polynomials with respect to an integral scalar product associated with a probability measure supported on the real line. Subsequently, for the matrix case, this topic was explored in a series of papers \cite{Yakhlef1,Yakhlef2,Yakhlef3}, where the corresponding Christoffel--Geronimus--Uvarov formula for perturbed polynomials was derived, specifically when a sole Dirac delta supported on a point of the real line is added to a matrix of measures. Now, we extend this analysis to the general case of an additive perturbation with a discrete finite support in the $y$-variable of a sesquilinear form. Therefore, we allow to deal with  additive perturbations that have, in the $y$-variable, an arbitrary finite number of derivatives of the Dirac delta at several different points, along with arbitrary linearly independent generalized functions in the $x$-variable.

\subsection{Additive perturbations}
Here, we present an intriguing formula for additive perturbations of matrix generalized kernels, which will play a crucial role in our discussion of the matrix Uvarov transformation.
\begin{pro}[Additive perturbation and reproducing kernels]\label{pro:additive}
Let consider an additive perturbation of the matrix of bivariate generalized functions $u_{x,y}$ as
$
\hat u_{x,y}= u_{x,y}+v_{x,y}$,
and assume that both $u_{x,y}$ and $\hat u_{x,y}$ are quasidefinite. Then, the following expressions hold:
\begin{align*}
	\hat P^{[1]}_{n}(z)&=
	P^{[1]}_{n}(z)-\big\langle\hat P^{[1]}_n(x),\big(K_{n-1}(z,y)\big)^\top\big\rangle_v,\\
	(\hat P^{[2]}_{n}(z))^\top&=(P^{[2]}_{n}(z))^\top-\big\langle K_{n-1}(x,z),\hat P^{[2]}_n(y)\big\rangle_v
\end{align*}
and
\begin{align*}
	\hat H_n&= H_n+\prodint{\hat P^{[1]}_n(x),P^{[2]}_n(y)}_{v}=H_n+\prodint{ P^{[1]}_n(x),\hat P^{[2]}_n(y)}_{v}.
\end{align*}
\end{pro}
\begin{proof}
	From \eqref{eq:orthogonality1} and \eqref{eq:orthogonality2}	we deduce
\begin{gather}
\label{eq:additive1}\begin{aligned}
\prodint{\hat P^{[1]}_n(x),P^{[2]}_m(y)}_{\hat{u}}&=0_p,&\prodint{P^{[1]}_m(x), \hat P^{[2]}_n(y)}_{\hat u}&= 0_p,  &m&\in\{1,\dots n-1\},
\end{aligned}\\
\label{eq:additive2}
\prodint{\hat P^{[1]}_n(x),P^{[2]}_n(y)}_{\hat u}= \hat H_n= \prodint{ P^{[1]}_n(x), \hat P^{[2]}_n(y)}_{\hat u}.
\end{gather}
Recalling \eqref{eq:CD kernel} we get
\begin{align*}
\prodint{\hat P^{[1]}_n(x),(K_{n-1}(z,y))^\top}_{{u}}&=-\prodint{\hat P^{[1]}_n(x),(K_{n-1}(z,y))^\top}_{v},\\\prodint{K_{n-1}(x,z), \hat P^{[2]}_n(y)}_{u}&=-\prodint{K_{n-1}(x,z), \hat P^{[2]}_n(y)}_{v}.
\end{align*}
But, notice that $\hat P^{[1]}_n(x)- P^{[1]}_n(x)$ and $\hat P^{[2]}_n(y)- P^{[2]}_n(y)$ have degree $n-1$ and, therefore, recalling \eqref{eq:reproducing}
we deduce
\begin{align*}
\hat P^{[1]}_n(z)- P^{[1]}_n(z)&=	\prodint{ \hat P^{[1]}_n(x)- P^{[1]}_n(x),(K_{n-1}(z,y))^\top}_u\\
&=\prodint{ \hat P^{[1]}_n(x),(K_{n-1}(z,y))^\top}_u=-\prodint{ \hat P^{[1]}_n(x),(K_{n-1}(z,y))^\top}_v,\\
\big(\hat P^{[2]}_n(z)- P^{[2]}_n(z)\big)^\top&=	\prodint{ K_{n-1}(x,z),\hat P^{[2]}_n(y)- P^{[2]}_n(y)}_u\\
&=	\prodint{ K_{n-1}(x,z),\hat P^{[2]}_n(y)}_u=-	\prodint{ K_{n-1}(x,z),\hat P^{[2]}_n(y)}_v.
\end{align*}

Finally, from \eqref{eq:additive2} we get
\begin{align*}
\hat H_n&=\prodint{\hat P^{[1]}_n(x),P^{[2]}_n(y)}_{\hat u}=\prodint{\hat P^{[1]}_n(x),P^{[2]}_n(y)}_{ u}+\prodint{\hat P^{[1]}_n(x),P^{[2]}_n(y)}_{v}\\
&= H_n+\prodint{\hat P^{[1]}_n(x),P^{[2]}_n(y)}_{v},
\end{align*}
as well as
\begin{align*}
\hat H_n&=\prodint{ P^{[1]}_n(x),\hat P^{[2]}_n(y)}_{\hat u}=\prodint{ P^{[1]}_n(x),\hat P^{[2]}_n(y)}_{ u}+\prodint{ P^{[1]}_n(x),\hat P^{[2]}_n(y)}_{v}\\
&= H_n+\prodint{ P^{[1]}_n(x),\hat P^{[2]}_n(y)}_{v}.
\end{align*}
\end{proof}

\subsection{Matrix Christoffel--Uvarov formulas for Uvarov additive perturbations}

For functionals $\big(\beta^{(j)}_m\big)_x\in\big(\mathbb C^{p\times p}[x]\big)'$, we consider the following additive Uvarov  perturbation
\begin{align}\label{eq:v_uvarov_general}
\hat u_{x,y}&=u_{x,y}+v_{x,y}, &
v_{x,y}&=\sum_{j=1}^q\sum_{m=0}^{\kappa^{(j)}-1}\frac{(-1)^m}{m!}\big(\beta^{(j)}_m\big)_x\otimes\delta^{(m)}(y-x_j),
\end{align}
which has a finite support on the $y$ variable at the set $\{x_j\}_{j=1}^q$.
The set of  linear functionals $\big\{\big(\beta^{(j)}_m\big)_x\big\}_{\substack{j=1,\dots,q\\
		m=0,\dots,\kappa^{(j)}-1}}$ is supposed to be linearly independent in the bimodule $\big(\mathbb C^{p\times p}[x]\big)'$, i.e., for $ X^{(j)}_m \in\mathbb C^{p\times p}$ the unique solution to
$	\sum_{j=1}^q\sum_{m=0}^{\kappa^{(j)}-1}\big(\beta^{(j)}_m\big)_xX^{(j)}_m =0$
is $X^{(j)}_m=0_p$,
and for $Y^{(j)}_m \in\mathbb C^{p\times p}$ the unique solution to
$
	\sum_{j=1}^q\sum_{m=0}^{\kappa^{(j)}-1} Y^{(j)}_m\big(\beta^{(j)}_m\big)_x=0$
is $Y^{(j)}_m=0_p$.

We will assume along this subsection that both $u_{x,y}$ and $u_{x,y}+v_{x,y}$ are quasidefinite matrices of bivariate generalized functions.

\begin{defi}
We will define the degree of the Uvarov perturbation as $N=\kappa^{(1)}+\dots+\kappa^{(q)}$. The spectral jet associated with the finite support matrix of linear functionals $v_{x,y}$ is, for any sufficiently smooth matrix function $f(x)$ defined in an open set in $\mathbb R$ containing the support $\{x_1,\dots,x_q\}$, the following matrix:
	\begin{align*}
	\mathcal J_f=\begin{bmatrix}
f(x_1)&\Cdots&\dfrac{(f(x))_{x_1}^{(\kappa^{(1)}-1)}}{(\kappa^{(1)}-1)!}&\Cdots&f(x_q)&\cdots&\dfrac{(f(x))_{x_q}^{(\kappa^{(q)}-1)}}{(\kappa^{(q)}-1)!}
	\end{bmatrix}\in\mathbb C^{p\times Np}.
	\end{align*}
	For a  matrix of kernels  $K(x,y)$, we have
		\begin{align*}
		\mathcal J^{[0,1]}_K(x)=\begin{bmatrix}
		K(x,x_1)\\\vdots\\\dfrac{(K(x,y))_{x,x_1}^{(0,\kappa^{(1)}-1)}}{(\kappa^{(1)}-1)!}\\\vdots\\K(x,x_q)\\\vdots\\\dfrac{(K(x,y))_{x,x_q}^{(0,\kappa^{(q)}-1)}}{(\kappa^{(q)}-1)!}
		\end{bmatrix}\in\mathbb C^{Np\times p}.
		\end{align*}
		We also require  the introduction of the following matrices in $\mathbb C^{ p\times Np}$
{		 \begin{gather*}
 \prodint{
 	P(x),(\beta)_x}	:=
 \begin{bmatrix}
 \prodint{P(x),\big(\beta^{(1)}_{0}\big)_x }&\dots&
 \prodint{P(x),\big(\beta^{(q)}_{\kappa^{(q)}-1}\big)_x}
 \end{bmatrix}.
 \end{gather*}}
\end{defi}

\begin{pro}\label{pro:uvarov0}
	The following relations
				\begin{align*}
				\prodint{\hat P^{[1]}_n(x), (K_{n-1}(z,y))^\top}_v&= \prodint{\hat P_n^{[1]}(x),(\beta)_x}\mathcal  J^{[0,1]}_{K_{n-1}}(z),\\
				\prodint{K_{n-1}(x,z), \hat P^{[2]}_n(y)}_v&=
				\prodint{K_{n-1}(x,z),(\beta)_x}\big(\mathcal  J_{\hat P^{[2]}_n}\big)^\top
				\end{align*}
				hold.
\end{pro}
\begin{proof}
Direct substitution gives
	\begin{align*}
	\prodint{\hat P^{[1]}_n(x), (K_{n-1}(z,y))^\top}_v&= \sum_{j=1}^q\sum_{m=0}^{\kappa^{(j)}-1}\frac{1}{m!}
	\prodint{\hat P_n^{[1]}(x),\big(\beta^{(j)}_m\big)_x}\Big(K_{n-1}(z,x)\Big)^{(m)}_{x_j},\\
	\prodint{K_{n-1}(x,z), \hat P^{[2]}_n(y)}_v&= \sum_{j=1}^q\sum_{m=0}^{\kappa^{(j)}-1}\frac{1}{m!}	\prodint{K_{n-1}(x,z),\big(\beta^{(j)}_m\big)_x}
	\Big(\big(\hat P^{[2]}_n(y)\big)^\top\Big)^{(m)}_{x_j},
	\end{align*}
	and the result follows.	
\end{proof}

\begin{defi}\label{def:kernels}
	For $j\in\{1,\dots,q\},
	m\in\{1,\dots,
	\kappa^{(j)}-1\}$, let us define the left 
	kernel subspace
	\begin{align*}
	\operatorname{Ker}^R_\beta:=\Big\{P(x)\in\mathbb C^{p\times p}[x]:
	\prodint{\big(\beta^{(j)}_m\big)_x,P(x)}=0_p\Big\},
	\end{align*}
	and the bilateral ideal
	\begin{align*}
				\mathbb I :=(x-x_1)^{\kappa^{(1)}}\cdots (x-x_q)^{\kappa^{(q)}}\mathbb C^{p\times p}[x].
				\end{align*}
The corresponding	right and left
	orthogonal complements with respect to the sesquilinear form are
	\begin{align*}
\big(	\operatorname{Ker}^R_\beta\big)^{\perp_u^R}&:=\Big\{
Q(x)\in\mathbb C^{p\times p}[x]: \prodint{P(x),Q(y)}_u=0_p, P(x)\in\operatorname{Ker}^R_\beta\Big\},
\\
\mathbb I^{\perp_u^R}&:=\Big\{
Q(y)\in\mathbb C^{p\times p}[y]: \prodint{Q(x),P(y)}_u=0_p, P(x)\in\mathbb I\Big\}.
	\end{align*}
\end{defi}

Given a block matrix $M=\begin{bNiceMatrix}[small]
	A & B\\
	C &D
\end{bNiceMatrix}$
with  $A\in\C^{N\times N}$, $B\in \C^{N\times M} $,$C\in\C^{M\times N}$, $D\in \C^{M\times M} $ and $ \det A\neq 0$ we define the quasi-determinant
$\Theta_*(M):= D-B A^{-1} C$.

\begin{teo}[Christoffel--Uvarov formulas]\label{teo:uvarov}
If either one of the two conditions $\big(	\operatorname{Ker}^R_\beta\big)^{\perp_u^R}=\{0_p\}$ or $\mathbb I^{\perp_u^L}=\{0_p\}$ is satisfied, then the matrix $I_{Np}+\prodint{\mathcal  J^{[0,1]}_{K_{n}}(x),(\beta)_x}$ is nonsingular, and the Uvarov perturbed matrix orthogonal polynomials and $H$'s matrices can be expressed as follows:
	\begin{align*}
	\hat P^{[1]}_n(x)&=\Theta_*\begin{bmatrix}
	I_{Np}+\prodint{\mathcal  J^{[0,1]}_{K_{n-1}}(x),(\beta)_x}&\mathcal  J^{[0,1]}_{K_{n-1}}(x)\\
\prodint{ P_n^{[1]}(x),(\beta)_x}& P_n^{[1]}(x)
	\end{bmatrix}, \\
	(\hat P^{[2]}_n(y))^\top&=\Theta_*\begin{bmatrix}
I_{Np}+\prodint{\mathcal  J^{[0,1]}_{K_{n-1}}(x),(\beta)_x}&\big(\mathcal J_{ P^{[2]}_{n}}\big)^\top\\
\prodint{K_{n-1}(x,y),(\beta)_x}	& (P_n^{[2]}(y ))^\top
	\end{bmatrix},\\
	\hat H_n&=\Theta_*\begin{bmatrix}
	I_{Np}+\prodint{\mathcal  J^{[0,1]}_{K_{n-1}}(x),(\beta)_x}&- (\mathcal J_{P^{[2]}_n}  )^\top\\
\prodint{ P_n^{[1]}(x),(\beta)_x}& H_n
	\end{bmatrix}.
	\end{align*}
\end{teo}

\begin{proof}
From Propositions \ref{pro:additive} and  \ref{pro:uvarov0} we get
\begin{align}\label{eq:uvarov11}
\hat P^{[1]}_{n}(z)&=
P^{[1]}_{n}(z)-\prodint{\hat P_n^{[1]}(x),(\beta)_x}\mathcal  J^{[0,1]}_{K_{n-1}}(z),\\
(\hat P^{[2]}_{n}(z))^\top&=(P^{[2]}_{n}(z))^\top-
\prodint{K_{n-1}(x,z),(\beta)_x}\big(\mathcal  J_{\hat P^{[2]}_n}\big)^\top,
\label{eq:uvarov12}
\end{align}
and, as a consequence,
\begin{align}
\label{eq:nonsing1uvarov}
\prodint{\hat P^{[1]}_{n}(x),(\beta)_x}\Big(I_{Np}+\prodint{\mathcal  J^{[0,1]}_{K_{n-1}}(x),(\beta)_x}\Big)&=
\prodint{ P^{[1]}_{n}(x),(\beta)_x},\\\label{eq:nonsing2uvarov}
\Big(I_{Np}+\prodint{\mathcal  J^{[0,1]}_{K_{n-1}},(\beta)_x}\Big)\big(\mathcal  J_{\hat P^{[2]}_n}\big)^\top
&=\big(\mathcal  J_{ P^{[2]}_n}\big)^\top.
\end{align}

Let us check the nonsingularity of  the matrices $I_{Np}+\prodint{\mathcal  J^{[0,1]}_{K_{n-1}}(x),(\beta)_x}$.  If we assume the contrary, then we can find a nonzero vector $ X\in \mathbb C^{Np}$ such that
\begin{align*}
\Big(I_{Np}+\prodint{\mathcal  J^{[0,1]}_{K_{n-1}}(x),(\beta)_x}\Big) X=0,
\end{align*}
is the zero vector in $\mathbb C^{Np}$, and, equivalently,  a covector $Y\in(\mathbb C^{Np})^*$ with
\begin{align*}
Y\Big(I_{Np}+\prodint{\mathcal  J^{[0,1]}_{K_{n-1}}(x),(\beta)_x}\Big)=0.
\end{align*}

Thus, using \eqref{eq:nonsing1uvarov} and \eqref{eq:nonsing2uvarov} we conclude that $\prodint{ P^{[1]}_{n}(x),(\beta)_x}X=0$ and $Y\big(\mathcal  J_{\hat P^{[2]}_n}\big)^\top=0$. The definition of the Christoffel--Darboux kernels \eqref{eq:CD kernel} implies
\begin{multline*}
I_{Np}+\prodint{\mathcal  J^{[0,1]}_{K_{n}}(x),(\beta)_x}=I_{Np}+\prodint{\mathcal  J^{[0,1]}_{K_{n-1}}(x),(\beta)_x}\\+
(\mathcal J_{P_n^{[2]}})^\top
(H_n)^{-1}
\prodint{ P^{[1]}_{n}(x),(\beta)_x},
\end{multline*}
and, consequently, we deduce that $\Big(I_{Np}+\prodint{\mathcal  J^{[0,1]}_{K_{n}}(x),(\beta)_x}\Big)X=0$ and $Y\Big(I_{Np}+\prodint{\mathcal  J^{[0,1]}_{K_{n}}(x),(\beta)_x}\Big)=0$, so that $\prodint{ P^{[1]}_{n+1}(x),(\beta)_x}X=0$ and $Y(\mathcal J_{P_{n+1}^{[2]}})^\top=0$, and so forth and so on.
Thus, we deduce that $\prodint{ P^{[1]}_{k}(x),(\beta)_x}X=0$ and $Y(\mathcal J_{P_k^{[2]}})^\top=0$ for $k\in\{n,n+1,\dots\}$ . If we write
\begin{align*}
X&=\begin{bmatrix}
X^{(1)}_{0}&
\cdots&
X^{(1)}_{\kappa^{(1)-1}}&
\cdots&
X^{(q)}_{0}&
\cdots&
X^{(q)}_{\kappa^{(q)-1}}
\end{bmatrix}^\top,\\
Y&=\begin{bmatrix}
Y^{(1)}_{0}&\Cdots&
Y^{(1)}_{\kappa^{(1)-1}}&\Cdots&
Y^{(q)}_{0}&\Cdots&
Y^{(q)}_{\kappa^{(q)-1}}
\end{bmatrix},
\end{align*}
 then the linear functionals
\begin{align*}
\beta\cdot X:=\sum_{j=1}^q\sum_{m=0}^{\kappa^{(j)}-1}\big(\beta^{(j)}_m\big)_xX^{(j)}_m,&
Y\cdot\delta:=\sum_{j=1}^q\sum_{m=0}^{\kappa^{(j)}-1}Y^{(j)}_m\frac{(-1)^m}{m!}\delta^{(m)}(y-x_j),
\end{align*}
are  such that
$\langle P^{[1]}_k(x),\beta\cdot X\rangle=0$ and $\prodint{Y\cdot\delta, \big(P^{[2]}_{k}(x)\big)^\top}=0$,
for $k\in\{n,n+1,\dots\}$. We can say that
\begin{align*}
\beta\cdot X&\in \big(\big\{P^{[1]}_k(x)\big\}_{k=n}^\infty\big)^{\perp_R}, &
Y\cdot \delta&\in \Big(\Big\{\big(P^{[2]}_k(x)\big)^\top\Big\}_{k=n}^\infty\Big)^{\perp_L}.
\end{align*}
It is convenient at this point to  recall that  the  topological and algebraic duals of the set of matrix polynomials coincide, i.e., $(\mathbb C^{p\times p}[x])'=(\mathbb C^{p\times p}[x])^*=\mathbb C^{p\times p}[\![x]\!]$, where we understand the set of matrix polynomials or matrix formal series as left or right modules over the ring of matrices. We also recall that the module of matrix polynomials of  degree less than or equal to $m$ is a the free module of rank $m+1$.  Therefore, for each positive integer $m$, we consider the  basis in the above dual space given by the following set of
matrices of linear functionals $\big\{(P^{[1]}_k)^*\big\}_{k=0}^m$,  dual to $\big\{P^{[1]}_k(x)\big\}_{k=0}^m.$
$ \prodint{P^{[1]}_k(x), (P^{[1]}_l)^*}=\delta_{k,l} I_p.$ As a consequence,
any linear functional can be written
$\sum\limits_{k=0}^m (P^{[1]}_l)^* C_k$,
where $C_k\in\mathbb C^{p\times p}$.  Then, $\big(\big\{P^{[1]}_k(x)\big\}_{k=n}^\infty\big)^{\perp_R}=\big\{(P^{[1]}_k)^*\big\}_{k=0}^{n-1}\mathbb C^{p\times p}\cong  \big(\mathbb C^{p\times p}\big)^n$ and, therefore, is a free right module of rank $m+1$.
But,  according to \eqref{eq:biorthogonal} and Schwartz kernel theorem, for
the set of matrices of linear functionals $\big\{\mathcal L_u\big(P^{[2]}_k(y)\big)\big\}_{k=0}^{n-1}\subseteq \big(\big\{P^{[1]}_k(x)\big\}_{k=n}^\infty\big)^{\perp_R}=\big\{(P^{[1]}_k)^*\big\}_{k=0}^{n-1}\mathbb C^{p\times p}$, and we conclude
$\big\{\mathcal L_u\big(P^{[2]}_k(x)\big)\big\}_{k=0}^{n-1}\mathbb C^{p\times p}= \big(\big\{P^{[1]}_k(x)\big\}_{k=n}^\infty\big)^{\perp_R}$.
Thus, we can write $\beta\cdot X=\mathcal L_u \big(Q_X(x)\big)$, where $Q_X(x)$ is  a matrix polynomial with degree $\deg (Q_X(x))\leq n-1$. A similar argument leads us to write $Y\cdot\delta=\mathcal L_u' \big(_YQ(x)\big)$, where $\mathcal L_u'$ denotes the transpose operator of $\mathcal L_u$ --see \cite{Schwartz1}--   and $Q_Y(x)$ is  a matrix polynomial with degree $\deg (Q_Y(x))\leq n-1$.
For
\begin{align*}
 P(x)&\in \operatorname{Ker}^R(\beta\cdot X):=\{P(x)\in\mathbb C^{p\times p}[x]: \prodint{\beta_X,P(x)}=0\},
	\\
 P(y)&\in \operatorname{Ker}^L(Y\cdot \delta):=\{P(y)\in\mathbb C^{p\times p}[y]: \prodint{P(y),Y\cdot\delta}=0\},
\end{align*}
we  have
$\big\langle\mathcal L_u \big(Q_X(y)\big),P(x)\big\rangle=0$ and $\big\langle P(y), \mathcal L_u'\big(Q_Y(x)\big)\big\rangle=0$,
and, for $ P(x)\in\operatorname{Ker}^R(\beta\cdot X), P(y)\in\operatorname{Ker}^L(Y\cdot\delta)$, the Schwartz kernel theorem gives
$\langle P(x), Q_X(y) \rangle_u=0$ and $\langle Q_Y(x) ,P(y)\rangle_u=0$,
which in turn implies
$Q_X(x)\in\big(\operatorname{Ker}^R(\beta\cdot X)\big))^{\perp_u^R}$ and $Q_Y(x)\in\big(\operatorname{Ker}^L(Y\cdot \delta)\big))^{\perp_u^L}$.
Notice that $\operatorname{Ker}^R_\beta\subset \operatorname{Ker}^R(\beta\cdot X)$
and, consequently,
$(\operatorname{Ker}^R_\beta)^{\perp^R_u}\supset (\operatorname{Ker}^R(\beta\cdot X))^{\perp^R_u}$. But we have assumed that $\big(\operatorname{Ker}^R_\beta\big)^{\perp_u^R}=\{0_p\}$, so that $Q_X(x)=0_p$ and  $\beta\cdot X=0_p$. Consequently, since the set $\big\{\big(\beta^{(j)}_m\big)_x\big\}_{\substack{j=1,\dots,q\\
		m=0,\dots,\kappa^{(j)}-1}}$ is linearly independent in the bimodule $\big(\mathbb C^{p\times p}[x]\big)'$, we deduce that $X=0$, which contradicts  the initial assumption.
		Observe that
		$\mathbb I \subset \operatorname{Ker}^L(Y\cdot\delta)$ so that $\big(\operatorname{Ker}^L(Y\cdot\delta)\big)^{\perp_u^L}\subset \mathbb I^{\perp_u^L}$. If we assume $\mathbb I^{\perp_u^L}=\{0_p\},$ then we get $Q_Y(x)=0_p$ and $Y=0$, in contradiction with the initial assumption and the matrix is
		$I_{Np}+\prodint{\mathcal  J^{[0,1]}_{K_{n-1}}(x),(\beta)_x}$. This also implies that $Y=0$

Hence, the matrix  $I_{Np}+\prodint{\mathcal  J^{[0,1]}_{K_{n}}(x),(\beta)_x}$ is nonsingular, allowing us to clean \eqref{eq:nonsing1uvarov} and \eqref{eq:nonsing2uvarov} and get
\begin{align*}
\prodint{\hat P^{[1]}_{n}(x),(\beta)_x}&=
\prodint{ P^{[1]}_{n}(x),(\beta)_x}\Big(I_{Np}+\prodint{\mathcal  J^{[0,1]}_{K_{n-1}}(x),(\beta)_x}\Big)^{-1},\\
\big(\mathcal  J_{ \hat P^{[2]}_n}\big)^\top
&=\Big(I_{Np}+\prodint{\mathcal  J^{[0,1]}_{K_{n-1}}(x),(\beta)_x}\Big)^{-1}\big(\mathcal  J_{ P^{[2]}_n}\big)^\top.
\end{align*}
These relations, when introduced in \eqref{eq:uvarov11} and \eqref{eq:uvarov12}, yield
{\small\begin{align*}
\hat P^{[1]}_{n}(x)&=
P^{[1]}_{n}(x)-\prodint{ P^{[1]}_{n}(x),(\beta)_x}\Big(I_{Np}+\prodint{\mathcal  J^{[0,1]}_{K_{n-1}}(x),(\beta)_x}\Big)^{-1}\mathcal  J^{[0,1]}_{K_{n-1}}(x),\\
(\hat P^{[2]}_{n}(y))^\top&=(P^{[2]}_{n}(y))^\top-
\prodint{K_{n-1}(x,y),(\beta)_x}
\Big(I_{Np}+\prodint{\mathcal  J^{[0,1]}_{K_{n-1}}(x),(\beta)_x}\Big)^{-1}\big(\mathcal  J_{ P^{[2]}_n}\big)^\top,
\end{align*}}
and the result follows.
To complete the proof notice that, see Proposition  \ref{pro:additive},
\begin{align*}
\hat H_n&=H_n+\prodint{\hat P^{[1]}_n(x),\beta }\big(\mathcal J_{P_n^{[2]}}\big)^\top\\
&=H_n+\prodint{ P^{[1]}_{n}(x),(\beta)_x}\Big(I_{Np}+\prodint{\mathcal  J^{[0,1]}_{K_{n-1}}(x),(\beta)_x}\Big)^{-1}\big(\mathcal J_{P_n^{[2]}}\big)^\top.
\end{align*}
\end{proof}

\subsection{Applications}
We discuss two particular cases of the general additive Uvarov perturbation presented above.

\subsubsection{Total derivatives}
We take the perturbation, which is supported by the diagonal $y=x$, in the following way
\begin{align}\label{eq:v_uvarov_diagonal_I}
v_{x,x}&=\sum_{j=1}^q\sum_{m=0}^{\kappa^{(j)}-1}\frac{(-1)^m}{m!}\beta^{(j)}_m\delta^{(m)}(x-x_j), & \beta^{(j)}_m&\in\mathbb C^{p\times p}.
\end{align}

\begin{pro}\label{pro:diagonal_masses}
	The discrete Hankel  mass terms \eqref{eq:v_uvarov_diagonal_I} are particular cases of  \eqref{eq:v_uvarov_general}  with
	\begin{align*}
	\big(\beta^{(j)}_{k}\big)_x=\sum_{n=0}^{\kappa^{(j)}-1-k}(-1)^{n}\frac{\beta^{(j)}_{k+n}}{(n)!}\delta^{(n)}(x-x_j).
	\end{align*}
\end{pro}

To discuss this reduction it is convenient  to introduce some further notation.
\begin{defi}
	For the family of perturbation matrices $\beta^{(j)}_m\in\mathbb C^{p\times p}$, let us consider
{	\begin{align*}
\begin{aligned}
		\beta^{(j)}:=\begin{bNiceMatrix}[small]
	\beta^{(j)}_0& \beta^{(j)}_1& \beta^{(j)}_2& \cdots&&&& \beta^{(j)}_{\kappa^{(j)}-1}\\
	\beta^{(j)}_1&  \beta^{(j)}_2& &&&& \beta^{(j)}_{\kappa^{(j)}-1}&0_p\\
	\beta^{(j)}_2& &&&& \beta^{(j)}_{\kappa^{(j)}-1}&\Iddots[shorten-end=-3pt,shorten-start=-2pt]&\vdots\\
	\\\vdots&  & & \Iddots[shorten-end=-12pt]& &  &&
	\\&  & & & &  &&
	\\&  & & & &  &&
	\\
	\beta^{(j)}_{\kappa^{(j)}-1}&0_p&\cdots&&&&& 0_p
	\end{bNiceMatrix}\in\mathbb C^{\kappa^{(j)}p\times \kappa^{(j)}p}
\end{aligned}
	\end{align*}}
	and if $N:=\kappa^{(1)}+\dots+\kappa^{(q)}$, then  we introduce
	\begin{align*}
	\beta=\diag(\beta^{(1)},\dots,\beta^{(q)})\in\mathbb C^{ Np\times Np}.
	\end{align*}
		We also introduce some additional jets. First a spectral jet with respect to the first variable
		{	\scriptsize	\begin{align*}
				\mathcal J^{[1,0]}_K(y)=\begin{bmatrix}
				K(x_1,y)&\dots&\dfrac{(K(x,y))_{x_1,y}^{(\kappa^{(1)}-1,0)}}{(\kappa^{(1)}-1)!}&\dots&K(x_q,y)&\dots&\dfrac{(K(x,y))_{x_q,y}^{(\kappa^{(q)}-1,0)}}{(\kappa^{(q)}-1)!}
				\end{bmatrix}\in\mathbb C^{p\times Np},
				\end{align*}		}
		and also a  double spectral jet of a matrix kernel
	{\scriptsize	\begin{align*}
	\hspace*{-2cm}	\begin{aligned}
			\mathcal J_K=\begin{bmatrix}
		K(x_1,x_1)&\cdots&\dfrac{(K(x,y))_{x_1,x_1}^{(\kappa^{(1)}-1,0)}}{(\kappa^{(1)}-1)!}&\cdots&K(x_q,x_1)&\cdots&\dfrac{(K(x,y))_{x_q,x_1}^{(\kappa^{(q)}-1,0)}}{(\kappa^{(q)}-1)!}
		\\
		\vdots&&\vdots&&\vdots& &\vdots \\
		\dfrac{(K(x,y))_{x_1,x_1}^{(0,\kappa^{(1)}-1)}}{(\kappa^{(1)}-1)!}&\cdots &\dfrac{(K(x,y))_{x_1,x_1}^{(\kappa^{(1)}-1,\kappa^{(1)}-1)}}{(\kappa^{(1)}-1)!(\kappa^{(1)}-1)!}&\cdots&	\dfrac{(K(x,y))_{x_q,x_1}^{(0,\kappa^{(1)}-1)}}{(\kappa^{(1)}-1)!}&\cdots &\dfrac{(K(x,y))_{x_1,x_q}^{(\kappa^{(q)}-1,\kappa^{(1)}-1)}}{(\kappa^{(q)}-1)!(\kappa^{(1)}-1)!}	\\
		\vdots&&\vdots&&\vdots& &\vdots \\
		K(x_1,x_q)&\cdots&\dfrac{(K(x,y))_{x_1,x_q}^{(\kappa^{(1)}-1,0)}}{(\kappa^{(1)}-1)!}&\cdots&K(x_q,x_q)&\cdots&\dfrac{(K(x,y))_{x_q,x_q}^{(\kappa^{(q)}-1,0)}}{(\kappa^{(q)}-1)!}\\
		\vdots&&\vdots&&\vdots& &\vdots \\
		\dfrac{(K(x,y))_{x_1,x_q}^{(0,\kappa^{(q)}-1)}}{(\kappa^{(q)}-1)!}&\cdots &\dfrac{(K(x,y))_{x_1,x_q}^{(\kappa^{(1)}-1,\kappa^{(q)}-1)}}{(\kappa^{(1)}-1)!(\kappa^{(q)}-1)!}&\cdots&	\dfrac{(K(x,y))_{x_q,x_q}^{(0,\kappa^{(q)}-1)}}{(\kappa^{(q)}-1)!}&\cdots &\dfrac{(K(x,y))_{x_q,x_q}^{(\kappa^{(q)}-1,\kappa^{(q)}-1)}}{(\kappa^{(q)}-1)!(\kappa^{(q)}-1)!}
		\end{bmatrix},
		\end{aligned}
		\end{align*}}
		which belongs to $\mathbb C^{Np\times Np}$.
		We have employed the compact notation
		\begin{align*}
		(K(x,y))^{(n,m)}_{a,b}=\frac{\partial^{n+m} K}{\partial x^n\partial y^m}\Big|_{x=a,y=b}.
		\end{align*}
\end{defi}

\begin{pro}	
When  the mass term $v_{x,y}$ is as \eqref{eq:v_uvarov_diagonal_I} the triviality of kernel subspace
$\operatorname{Ker}^R_\beta=\{0_p\}$
	is ensured whenever
$	\mathbb I^{\perp_u^R}=\{0_p\}$.
	Thus, the quasideterminantal expressions of Theorem \ref{teo:uvarov} hold whenever $\mathbb I^{\perp_u^L}=\mathbb I^{\perp_u^R}=\{0_p\}$.
	Moreover, if the generalized kernel $u_{x,y}$ is of Hankel type, then  the quasidefiniteness of $u$ ensures the trivially  of the left and right  orthogonal complements of the bilateral ideal $\mathbb I$.
\end{pro}
\begin{proof}
For the diagonal case we  choose the mass terms as in Proposition \ref{pro:diagonal_masses}
\begin{align*}
\big(\beta^{(j)}_{k}\big)_x=\sum_{n=0}^{\kappa^{(j)}-1-k}(-1)^{n}\frac{\beta^{(j)}_{k+n}}{(n)!}\delta^{(n)}(x-x_j).
\end{align*}
Thus, $\mathbb I\subset\operatorname{Ker}^R_\beta$,
and, consequently, $
\mathbb I^{\perp_u^R}\supset \big(\operatorname{Ker}^R_\beta\big)^{\perp_u^R}$.
If $u_{x,y}$ is of Hankel type, then $P(x)\in \mathbb I^{\perp_u^R}$ if
$\prodint{
	P(x)(x-x_1)^{\kappa^{(1)}}\cdots (x-x_q)^{\kappa^{(q)}},x^nI_p}_u=0$,
for $n\in\{0,1,\dots\}.$ Since  $u$ is quasidefinite, then $P(x)=0_p$,and we get $\mathbb I^{\perp_u^R}=\{0_p\}$.
A similar argument leads to  the triviality $\mathbb I^{\perp_u^L}=\{0_p\}$.
\end{proof}

\begin{pro}
When  the mass term $v_{x,y}$ is as \eqref{eq:v_uvarov_diagonal_I} the perturbed matrix orthogonal polynomials and $H$ matrices are
\begin{align*}
\hat P^{[1]}_n(x)&=\Theta_*\begin{bmatrix}
I_{Np}+\beta \mathcal J_K & \beta \mathcal J^{[0,1]}_{K_{n-1}}(x)\\
\mathcal J_{P^{[1]}_n} & P_n^{[1]}(x)
\end{bmatrix}, \\
(\hat P^{[2]}_n(y))^\top&=\Theta_*\begin{bmatrix}
I_{Np}+ \beta\mathcal J_K  &\beta (\mathcal J_{P^{[2]}_n}  )^\top\\
\mathcal J^{[1,0]}_{K_{n-1}}(y)	 & (P_n^{[2]}(y ))^\top
\end{bmatrix},\\
\hat H_n&=\Theta_*\begin{bmatrix}
I_{Np}+ \beta\mathcal J_K &- \beta(\mathcal J_{P^{[2]}_n}  )^\top\\\
\mathcal J_{P^{[1]}_n} & H_n
\end{bmatrix}.
\end{align*}	
\end{pro}

\subsubsection{Uvarov perturbations with finite discrete support}

We will assume that the matrices of generalized functions $(\beta^{(j)}_m)_x$ are supported on a discrete finite set, say $\{\tilde x_b\}_{b=1}^{\tilde q}$,
with multiplicities $\tilde\kappa^{(b)}$ such that $\tilde\kappa^{(1)}+\cdots+\tilde\kappa^{(\tilde q)}=\tilde N$,  so that
\begin{align}\label{eq:discrete_support}
(\beta^{(j)}_m)_x=\sum_{b=1}^{\tilde q}\sum_{l=0}^{\tilde \kappa^{(b)}-1}\beta^{(b,j)}_{l,m}\frac{(-1)^l}{l!}\delta^{(l)}(x-\tilde x_b).
\end{align}
For $\beta^{(b,j)}_{l,m}\in\mathbb C^{p\times p}$, the additive perturbation is
\begin{align*}
v_{x,y}&=\sum_{b=1}^{\tilde q}\sum_{l=0}^{\tilde \kappa^{(b)}-1}\sum_{a=1}^q\sum_{m=0}^{\kappa^{(a)}-1}\frac{(-1)^{l+m}}{l!m!}\beta^{(b,a)}_{l,m}\delta^{(l)}(x-\tilde x_b)\otimes\delta^{(m)}(y-x_a).
\end{align*}
This implies that there is a discrete support of the perturbing matrix of generalized functions, $\operatorname{supp}(v_{x,y})
=\{\tilde x_b\}_{b=1}^{\tilde q}\times  \{ x_j\}_{j=1}^{ q}$. The perturbed sesquilinear form  is
\begin{multline*}
\prodint{P(x),Q(y)}_{u+v}=\prodint{P(x),Q(y)}_u\\+
\sum_{b=1}^{\tilde q}\sum_{l=0}^{\tilde \kappa^{(b)}-1}\sum_{j=1}^q\sum_{m=0}^{\kappa^{(j)}-1}\frac{1}{l!m!}\big(P(x)\big)^{(l)}_{\tilde x_b}\beta^{(b,j)}_{l,m}\big((Q(y))^\top\big)^{(m)}_{x_j}.
\end{multline*}

\begin{defi}
	We introduce the following block rectangular matrix of couplings that belongs to $\in\mathbb C^{\tilde N p\times N p}$
{	\begin{align*}
\beta:=\begin{bmatrix}
\beta^{(0,0)}_{1,1}&\cdots &\beta^{(0,\kappa^{(1)}-1)}_{1,1}&\cdots&\beta^{(0,0)}_{1,q}&\cdots&\beta^{(0,\kappa^{(q)}-1)}_{1,q}\\
\vdots & & \vdots & & \vdots & & \vdots\\
\beta^{(\tilde \kappa^{(1)}-1,0)}_{1,1}&\cdots &\beta^{(\tilde \kappa^{(1)}-1,\kappa^{(1)}-1)}_{1,1}&\cdots&\beta^{(\tilde \kappa^{(1)}-1,0)}_{1,q}&\cdots&\beta^{(\tilde \kappa^{(1)}-1,\kappa^{(q)}-1)}_{1,q}\\
\vdots & & \vdots & & \vdots & & \vdots\\
\beta^{(0,0)}_{\tilde q,1}&\cdots &\beta^{(0,\kappa^{(1)}-1)}_{\tilde q,1}&\cdots&\beta^{(\tilde \kappa^{(1)}-1,0)}_{\tilde q,q}&\cdots&\beta^{(\tilde \kappa^{(1)}-1,\kappa^{(q)}-1)}_{\tilde q,q}\\
\vdots & & \vdots & & \vdots & & \vdots\\
\beta^{(\tilde\kappa^{(\tilde q)}-1,0)}_{\tilde q,1}&\cdots &\beta^{(\tilde\kappa^{(\tilde q)}-1,\kappa^{(1)}-1)}_{\tilde q,1}&\cdots&\beta^{(\tilde\kappa^{(\tilde q)}-1,0)}_{\tilde q,q}&\cdots&\beta^{(\tilde \kappa^{(\tilde q)}-1,\kappa^{(q)}-1)}_{\tilde q,q}
\end{bmatrix}
	\end{align*}}
	and,  given any matrix of kernels $K(x,y)$, we introduce the mixed double jet $\tilde{	\mathcal J}_K\in\mathbb C^{Np\times \tilde Np}$
{	\scriptsize\begin{align*}
\hspace*{-2cm}\begin{aligned}
	\tilde{	\mathcal J}_K:=\begin{bmatrix}
	K(\tilde x_1, x_1)&\cdots&\dfrac{(K(x,y))_{\tilde x_1, x_1}^{(\tilde  \kappa^{(1)}-1,0)}}{(\tilde \kappa^{(1)}-1)!}&\cdots&K(\tilde x_q, x_1)&\cdots&\dfrac{(K(x,y))_{\tilde  x_q,x_1}^{(\tilde \kappa^{(\tilde q)}-1,0)}}{(\tilde \kappa^{(\tilde q)}-1)!}
	\\
	\vdots&&\vdots&&\vdots& &\vdots \\
	\dfrac{(K(x,y))_{\tilde x_1,x_1}^{(0,\kappa^{(1)}-1)}}{(\kappa^{(1)}-1)!}&\cdots &\dfrac{(K(x,y))_{\tilde x_1,x_1}^{(\tilde \kappa^{(1)}-1,\kappa^{(1)}-1)}}{(\tilde \kappa^{(1)}-1)!(\kappa^{(1)}-1)!}&\cdots&	\dfrac{(K(x,y))_{\tilde x_q,x_1}^{(0,\kappa^{(1)}-1)}}{(\kappa^{(1)}-1)!}&\cdots &\dfrac{(K(x,y))_{\tilde  x_q,x_1}^{(\tilde  \kappa^{(\tilde q)}-1,\kappa^{(1)}-1)}}{(\tilde  \kappa^{(\tilde q)}-1)!(\kappa^{(1)}-1)!}	\\
	\vdots&&\vdots&&\vdots& &\vdots \\
	K(\tilde x_1,x_q)&\cdots&\dfrac{(K(x,y))_{\tilde x_1,x_q}^{(\tilde \kappa^{(1)}-1,0)}}{(\tilde \kappa^{(1)}-1)!}&\cdots&K(\tilde x_q,x_q)&\cdots&\dfrac{(K(x,y))_{\tilde x_q,x_q}^{(\tilde \kappa^{(\tilde q)}-1,0)}}{(\tilde \kappa^{(\tilde q)}-1)!}\\
	\vdots&&\vdots&&\vdots& &\vdots \\
	\dfrac{(K(x,y))_{\tilde x_1,x_q}^{(0,\kappa^{(q)}-1)}}{(\kappa^{(q)}-1)!}&\cdots &\dfrac{(K(x,y))_{\tilde  x_1,x_q}^{(\tilde \kappa^{(1)}-1,\kappa^{(q)}-1)}}{(\tilde  \kappa^{(1)}-1)!(\kappa^{(q)}-1)!}&\cdots&	\dfrac{(K(x,y))_{\tilde  x_q,x_q}^{(0,\kappa^{(q)}-1)}}{(\kappa^{(q)}-1)!}&\cdots &\dfrac{(K(x,y))_{\tilde x_q,x_q}^{(\tilde \kappa^{(\tilde q)}-1,\kappa^{(q)}-1)}}{(\tilde \kappa^{(\tilde q)}-1)!(\kappa^{(q)}-1)!}
	\end{bmatrix}.
\end{aligned}
	\end{align*}}
\end{defi}

With this election we get
\begin{pro}
For $\beta$'s as in \eqref{eq:discrete_support} we have
\begin{align*}
\begin{aligned}
	\prodint{P^{[1]}_n(x),\big(\beta\big)_x}&=\tilde{\mathcal J}_{P^{[1]}_n}\beta,&
\prodint{\mathcal  J^{[0,1]}_{K_{n-1}}(x),(\beta)_x}&=\tilde{\mathcal J}_{K_{n-1}} \beta,
\end{aligned}
\end{align*}
in terms of the spectral jet $\tilde {\mathcal J}$ relative to the  set $\{\tilde x_b\}_{b=1}^{\tilde q}$.
\end{pro}

\begin{coro}\label{coro:Uvarov}
	Whenever   one of the two conditions $\tilde{\mathbb I}^{\perp_u^R}=\{0_p\}$ or $\mathbb I^{\perp_u^L}=\{0_p\}$ holds, the matrix
	$	I_{Np}+\tilde{\mathcal J}_{K_{n-1}} \beta$ is nonsingular and the perturbed matrix orthogonal polynomials and $H$ matrices have the following quasideterminantal expressions
	\begin{align*}
	\hat P^{[1]}_n(x)&=\Theta_*\begin{bmatrix}[small]
	I_{Np}+\beta\tilde{\mathcal J}_{K_{n-1}} &\beta\mathcal  J^{[0,1]}_{K_{n-1}}(x)\\
	\tilde{\mathcal J}_{P^{[1]}_n},& P_n^{[1]}(x)
	\end{bmatrix}, \\
	(\hat P^{[2]}_n(y))^\top&=\Theta_*\begin{bmatrix}[small]
	I_{Np}+\beta\tilde{\mathcal J}_{K_{n-1}} &\beta\big(\mathcal J_{ P^{[2]}_{n}}\big)^\top\\
	\tilde{\mathcal J}^{[1,0]}_{K_{n-1}}(y)	& (P_n^{[2]}(y ))^\top
	\end{bmatrix},\\
	\hat H_n&=\Theta_*\begin{bmatrix}
	I_{Np}+\beta\tilde{\mathcal J}_{K_{n-1}} &- \beta(\mathcal J_{P^{[2]}_n}  )^\top\\
	\tilde{\mathcal J}_{P^{[1]}_n}& H_n
	\end{bmatrix}.
	\end{align*}
\end{coro}


\begin{thebibliography}{99}
	
%
%
%
%
%
%
%
%
%
%
%
%
%
%
%



		\bibitem{nuestro0} C. Álvarez-Fernández, G. Ariznabarreta, J. C. García-Ardila, M. Mañas	and F. Marcellán, \emph{Christoffel transformations for matrix orthogonal polynomials in the real line and the non-Abelian 2D Toda lattice hierarchy}, Int. Math. Res. Not. IMRN \textbf{2017} no. 5 (2017) 1285-1341.


		\bibitem{nuestro1} G. Ariznabarreta, J. C. García-Ardila, M. Mañas and F. Marcellán,
		\emph{Matrix biorthogonal polynomials on the real line: Geronimus transformations},
		Bull. Math. Sci. \textbf{9} (2019) 195007 (68 pages).
		
		\bibitem{nuestro2} G. Ariznabarreta, J. C. García-Ardila, M. Mañas and F. Marcellán,  \emph{Non-Abelian integrable hierarchies: matrix biorthogonal polynomials and perturbations}
		J. Phys. A: Math. Theor. \textbf{51} (2018) 205204.
	
	\bibitem{ari} G. Ariznabarreta and M. Mañas, \emph{Matrix orthogonal Laurent polynomials on the unit circle and Toda type integrable systems,} Adv. Math. \textbf{264} (2014) 396-463.
	
%
%
%
%
%
%
%
%
%
%
%
%
%
%
%
%
%
%
%
%
%
%
%
%
%
	
	\bibitem{DAS} D. Damanik, A. Pushnitski and B. Simon, \emph{The analytic theory of matrix orthogonal polynomials}, Surv. Approx. Theory \textbf{4} (2008) 1-85.
	
%
%
%
%
%
%
%
%
%
%
%
%
%
%
%
%
%
%
%
%
%
%
%
%
%
%
%
%
%
%
%
%
%
%
%
%
%
%
%
%
%
%
%
%
%
%
%
%
%
%
%
%
%
%
%
%
%
%
%
%
	
	
%
%
%
%
%
%
%
%
%
%
%
%
%
%
%
%
%
%
%
%
%
	
	\bibitem{rowen} L. Rowen, \emph{Ring Theory,} Vol. I,  Academic Press, San Diego CA, 1988.
	
%
%
%
%
	
	\bibitem{Schwartz1}  L. Schwartz,  \emph{Théorie des noyaux}, Proceedings of the International Congress of Mathematicians
	(Cambridge, MA, 1950), vol. 1 pp. 220-230,  Amer. Math. Soc. , Providence, RI, 1952.
	
%
%
	
	\bibitem{simon-cd} B. Simon, \emph{The Christoffel--Darboux Kernel} in \emph{Perspectives in Partial Differential Equations, Harmonic Analysis and Applications}, D. Mitrea and M. Mitrea (editors),  Proc. Sympos. Pure Math. \textbf{79} (2008) 295-346.
	
%
%
%
%
%
	
	\bibitem{Yakhlef1} H. O. Yakhlef, F. Marcellán and M. A. Piñar, \emph{Relative Asymptotics for Orthogonal Matrix Polynomials with Convergent Recurrence Coefficients}, J. Approx. Theory \textbf{111} (2001) 1-30.
	
	\bibitem{Yakhlef2} H. O. Yakhlef, F. Marcellán and M. A. Piñar, \emph{Perturbations in the Nevai matrix class of orthogonal matrix polynomials},  Linear Algebra Appl. \textbf{336} (2001) 231-254.
	
	\bibitem{Yakhlef3} H. O. Yakhlef and  F. Marcellán, \emph{Relative Asymptotics for  Matrix Orthogonal Polynomials for Uvarov Perturbations: The Degenerate Cases},  Mediterr. J. Math. \textbf{13} (2016) 3135–3153.
	
%
%
%
	
	\bibitem{Uva} V. B. Uvarov, \emph{The connection between systems of polynomials that are orthogonal with respect to different distribution functions,} USSR Comp.  Math. Phys. \textbf{9} (1969) 25-36.
	
%
%
\end{thebibliography}
\end{document}